\documentclass[11pt]{article}
\usepackage{amsmath,amssymb,amsthm,amscd}

\newtheorem{theorem}{Theorem}
\newtheorem{lemma}[theorem]{Lemma}

\newtheorem{corollary}[theorem]{Corollary}

\theoremstyle{plain}

\theoremstyle{definition}

\newtheorem{remark}[theorem]{Remark}

\renewcommand{\labelenumi}{\textup{(\theenumi)}}

\title{On one-sided topological conjugacy of topological Markov shifts and gauge actions on Cuntz--Krieger algebras
 \\
}
\author{Kengo Matsumoto \\
Department of Mathematics \\
Joetsu University of Education \\
Joetsu, 943-8512, Japan
}

\begin{document}

\date{2020, Oct 12}

\maketitle

\date{}

\def\det{{{\operatorname{det}}}}

\begin{abstract}
We will characterize  topological conjugacy classes of one-sided topological Markov shifts
in terms of the associated Cuntz--Krieger algebras and its gauge actions with potentials.
\end{abstract}





\newcommand{\Ker}{\operatorname{Ker}}
\newcommand{\sgn}{\operatorname{sgn}}
\newcommand{\Ad}{\operatorname{Ad}}
\newcommand{\ad}{\operatorname{ad}}
\newcommand{\orb}{\operatorname{orb}}

\def\Re{{\operatorname{Re}}}
\def\det{{{\operatorname{det}}}}
\newcommand{\K}{\mathcal{K}}

\newcommand{\N}{\mathbb{N}}
\newcommand{\C}{\mathcal{C}}
\newcommand{\R}{\mathbb{R}}
\newcommand{\T}{\mathbb{T}}
\newcommand{\Z}{\mathbb{Z}}
\newcommand{\Zp}{{\mathbb{Z}}_+}
\def\AF{{{\operatorname{AF}}}}

\def\OA{{{\mathcal{O}}_A}}
\def\OB{{{\mathcal{O}}_B}}
\def\OTA{{{\mathcal{O}}_{\tilde{A}}}}
\def\SOA{{{\mathcal{O}}_A}\otimes{\mathcal{K}}}
\def\SOB{{{\mathcal{O}}_B}\otimes{\mathcal{K}}}
\def\FA{{{\mathcal{F}}_A}}
\def\FB{{{\mathcal{F}}_B}}
\def\DA{{{\mathcal{D}}_A}}
\def\DB{{{\mathcal{D}}_B}}
\def\DZ{{{\mathcal{D}}_Z}}
\def\Ext{{{\operatorname{Ext}}}}
\def\Max{{{\operatorname{Max}}}}
\def\Per{{{\operatorname{Per}}}}
\def\PerB{{{\operatorname{PerB}}}}
\def\Homeo{{{\operatorname{Homeo}}}}
\def\HA{{{\frak H}_A}}
\def\HB{{{\frak H}_B}}
\def\HSA{{H_{\sigma_A}(X_A)}}
\def\Out{{{\operatorname{Out}}}}
\def\Aut{{{\operatorname{Aut}}}}
\def\Ad{{{\operatorname{Ad}}}}
\def\Inn{{{\operatorname{Inn}}}}
\def\det{{{\operatorname{det}}}}
\def\exp{{{\operatorname{exp}}}}
\def\nep{{{\operatorname{nep}}}}
\def\sgn{{{\operatorname{sign}}}}
\def\cobdy{{{\operatorname{cobdy}}}}
\def\Ker{{{\operatorname{Ker}}}}
\def\ind{{{\operatorname{ind}}}}
\def\id{{{\operatorname{id}}}}
\def\supp{{{\operatorname{supp}}}}
\def\co{{{\operatorname{co}}}}
\def\scoe{{{\operatorname{scoe}}}}
\def\coe{{{\operatorname{coe}}}}

\def\S{\mathcal{S}}

\def\coe{{{\operatorname{coe}}}}
\def\scoe{{{\operatorname{scoe}}}}
\def\uoe{{{\operatorname{uoe}}}}
\def\ucoe{{{\operatorname{ucoe}}}}
\def\event{{{\operatorname{event}}}}


\bigskip

For an irreducible non-permutation matrix $A =[A(i,j)]_{i,j=1}^N$ with entries in $\{0,1\},$  
let us denote by $(X_A,\sigma_A)$
the associated one-sided  topological Markov shift.
It consists of the compact Hausdorff space $X_A$ of right one-sided sequences
$(x_n)_{n\in \N}$ of $x_n \in \{1,2,\dots,N\}$ 
satsifying  $A(x_n, x_{n+1}) =1, n\in \N$, 
and the continuous surjective map of the right one-sided shift $\sigma_A: X_A\longrightarrow X_A$
defined by
$\sigma_A((x_n)_{n\in \N}) = (x_{n+1})_{n\in \N}$. 
The topology of $X_A$ is defined by the relative topology of the infinite product topology
of $\{1,2,\dots,N\}^\N$.
Two-sided topological Markov shift $(\bar{X}_A,\bar{\sigma}_A)$
is similarly defined by replacing right one-sided sequences $(x_n)_{n\in \N}$
with two-sided sequences $(x_n)_{n\in \Z}$.  
See the text books \cite{Kitchens}, \cite{LM} for general theory of symbolic dynamical systems.
By the monumental introduction of Cuntz--Krieger algebras
$\OA$ by J. Cuntz and W. Krieger in \cite{CK}, 
lots of important and interesting interplay between 
topological Markov shifts  and the Cuntz--Krieger algebras
have been studied and clarified.
The Cuntz--Krieger algebra $\OA$ for the matrix $A$ 
is defined by the universal unital $C^*$-algebra generated by 
$N$ partial isometries $S_1,\dots, S_N$
satisfying the relations:
$1 = \sum_{j=1}^N S_j S_j^*$ and $S_i^* S_i =\sum_{j=1}^NA(i,j)S_j S_j^*, i=1,2,\dots, N.$
As in the paper \cite{CK}, 
the original space $X_A$ appears in the algebra $\OA$ 
as a maximal commutative $C^*$-subalgebra written $\DA$
generated by projections of the form
$S_{\mu_1}\cdots S_{\mu_m}S_{\mu_m}^*\cdots S_{\mu_1}^*$
for $\mu_1,\dots,\mu_m \in \{1,2,\dots, N\}$
that is canonically isomorphic to the commutative $C^*$-algebra $C(X_A)$ 
of complex valued continuous functions on $X_A$,
through the identification between the projection
$S_{\mu_1}\cdots S_{\mu_m}S_{\mu_m}^*\cdots S_{\mu_1}^*$ 
and the characteristic function $\chi_{U_{\mu_1\cdots\mu_m}}$
on $X_A$ for the cylinder set
$U_{\mu_1\cdots\mu_m} =\{ (x_n)_{n\in \N} \in X_A \mid x_1 =\mu_1,\dots, x_m = \mu_m\}$.
The gauge action $\rho^A$ on $\OA$ is defined by the automorphisms
$\rho^A_t, t \in \R/\Z =\T$ satisfying   
$\rho^A_t(S_j) = \exp{(2\pi\sqrt{-1}t)}\cdot S_j, j=1,2,\dots,N$. 
 Cuntz and Krieger themselves proved in \cite{CK}  
 the following fundamental results (A), (B) and (C) 
that show close relationship between topological dynamical systems and $C^*$-algebras.
Let us denote by $\K$ and $\C$ the $C^*$-algebra of compact operators on the separable 
infinite dimensional Hilbert space $\ell^2(\N)$ and its maximal commutative $C^*$-subalgebra 
consisting of diagonal operators, respectively.
Let $A, B$ be two irreducible non-permutation matrices with entries in $\{0,1\}$.

\medskip

(A) If one-sided topological Markov shifts $(X_A,\sigma_A)$ and $(X_B,\sigma_B)$ are
topologically conjugate, then there exists an isomorphism
$\Phi: \OA\longrightarrow \OB$ of $C^*$-algebras such that 
\begin{equation}\label{eq:onesided}
\Phi(\DA) = \DB
\quad \text{ and }
\quad
\Phi\circ\rho^A_t = \rho^B_t\circ\Phi, \quad t \in \T.
\end{equation}
 
(B) If two-sided topological Markov shifts $(\bar{X}_A,\bar{\sigma}_A)$ 
and $(\bar{X}_B,\bar{\sigma}_B)$ are
topologically conjugate, then there exists an isomorphism
$\bar{\Phi}: \OA\otimes\K\longrightarrow \OB\otimes\K$ 
of $C^*$-algebras such that 
\begin{equation}\label{eq:twosided}
\bar{\Phi}(\DA\otimes\C) = \DB\otimes\C
\quad 
\text{  and }
\quad 
\bar{\Phi}\circ(\rho^A_t\otimes\id) 
= (\rho^B_t\otimes\id)\circ\bar\Phi, \quad t \in \T.
\end{equation} 

(C) If two-sided topological Markov shifts $(\bar{X}_A,\bar{\sigma}_A)$ 
and $(\bar{X}_B,\bar{\sigma}_B)$ are flow equivalent,
then there exists an isomorphism
$\bar{\Phi}: \OA\otimes\K\longrightarrow \OB\otimes\K$ 
of $C^*$-algebras such that 
\begin{equation}\label{eq:floweq}
\bar{\Phi}(\DA\otimes\C) = \DB\otimes\C.
\end{equation}

\medskip

The converse implications of the above three implications for each  have been 
longstanding open problems.
H. Matui and the author in \cite{MMKyoto} 
showed that the converse implication of (C) holds 
(cf. \cite{BCW}, \cite{CEOR}, \cite{CRS}, \cite{MaTAMS2018}, \cite{MMETDS}, etc.).
T. M. Carlsen and J. Rout in \cite{CR} showed 
that the converse implication of (B) holds by groupoid technique
(cf. \cite{CEOR}, \cite{CRS}, \cite{MaDyn2019}, etc.). 
Concerning the implication (A),
the author in \cite{MaPAMS2017} showed that 
the condition that  there exists an isomorphism
$\Phi: \OA\longrightarrow \OB$ of $C^*$-algebras satisfying
\eqref{eq:onesided}
is equivalent to the condition  that 
$(X_A,\sigma_A)$ and $(X_B,\sigma_B)$ are eventually conjugate,
where
one-sided topological Markov shifts 
$(X_A,\sigma_A)$ and $(X_B,\sigma_B)$ are said to be {\it eventually conjugate}\/
if there exist a homeomorphism $h:X_A\longrightarrow X_B$ and a nonnegative integer
$K$ such that 
\begin{equation}\label{eq:event}
\begin{cases}
\sigma_B^K(h(\sigma_A(x))) 
& = \quad \sigma_B^{K+1}(h(x)), \qquad x \in X_A, \\ 
\sigma_A^K(h^{-1}(\sigma_B(y))) 
& = \quad \sigma_A^{K+1}(h^{-1}(y)), \qquad y \in X_B. 
\end{cases}
\end{equation}
If one may take the integer $K$ as zero, then the relations
\eqref{eq:event} reduce to the definition of topological conjugacy 
$h: X_A\longrightarrow X_B$.
In \cite{BC}, K. A. Brix and Carlsen found
 an example of irreducible topological Markov shifts  
$(X_A, \sigma_A)$ and $(X_B, \sigma_B)$
 that are eventually conjugate, but not topologically conjugate.  
 In the paper, they  characterized topological conjugacy
  of one-sided topological Markov shifts 
  not only in terms of their associated \'etale groupoids (\cite[Corollary 3.5 (ii)]{BC})
  but also in terms of their Cuntz--Krieger algebras (\cite[Corollary 3.5 (iii)]{BC})
  in the following way.  
Following \cite{BC}, 
let $\tau_A:\OA\longrightarrow \OA$ be a completely positive map defined 
by $\tau_A(Y) = \sum_{i,j=1}^N S_i Y S_j^*, Y \in \OA$.
Brix and Carlsen proved that 
$(X_A, \sigma_A)$ and $(X_B, \sigma_B)$ are topologically conjugate
if and only if there exists an isomorphism $\varPhi: \OA\longrightarrow \OB$
of $C^*$-algebras satisfying $\varPhi(\DA) = \DB$ and $\varPhi\circ\tau_A = \tau_B\circ\varPhi$
(\cite[Corollary 3.5 ]{BC}).
This gives rise to a characterization of one-sided topological conjugacy
of one-sided topological Markov shifts in terms of $C^*$-algebra.
We note that the gauge action also appears in their other characterization 
of one-sided topological conjugacy as in \cite[Theorem 3.3 (iv)]{BC}.

In this short paper, we will attempt to characterize one-sided topological conjugacy
of one-sided topological Markov shifts
in terms of Cuntz--Krieger algebras and its gauge actions with potentials
to compare with the characterization of eventual conjugacy as in \eqref{eq:onesided}.
For an integer valued continuous function
 $g \in C(X_A,\Z)$ on $X_A$,
 the action $\rho^{A,g}$  is defined by the automorphisms $\rho^{A,g}_t, t \in \T$ 
 on $\OA$
satisfying
$\rho^{A, g}_t(S_j) = \exp{(2\pi\sqrt{-1}t g)}\cdot S_j, j=1,2,\dots,N$.
The action $\rho^{A,g}$ was called a generalized gauge action in \cite{MaJOT2015}, \cite{MaMZ2017}.
In this paper, we call it the gauge action with potential $g$.  
We will prove the following theorem.

\begin{theorem}\label{thm:main}
Let $A, B$ be two irreducible non-permutation matrices with entries in $\{0,1\}$.
The following assertions are equivalent:
\begin{enumerate}
\renewcommand{\theenumi}{\roman{enumi}}
\renewcommand{\labelenumi}{\textup{(\theenumi)}}
\item The one-sided topological Markov shifts $(X_A,\sigma_A)$ and $(X_B,\sigma_B)$ 
are topologically conjugate.
\item There exists an isomorphism
$\Phi: \OA\longrightarrow \OB$ of $C^*$-algebras such that 
$\Phi(\DA) = \DB$ and 
\begin{equation}\label{eq:gauge}
\Phi\circ\rho^{A, f\circ h}_t = \rho^{B, f}_t\circ\Phi
\qquad \text{ for all } \quad f \in C(X_B,\Z), \,  t \in \T,
\end{equation}
\end{enumerate}
where
$h: X_A\longrightarrow X_B$ is a homeomorphism induced by 
$\Phi: \DA\longrightarrow \DB$ satisfying
$\Phi(a) = a\circ h^{-1}$ for $a \in \DA$
under the canonical identification between $\DA$ and $C(X_A)$. 
\end{theorem}


\begin{proof}
(i) $\Longrightarrow$ (ii):
Suppose that there exists a topological conjugacy $h: X_A\longrightarrow X_B$
between $(X_A,\sigma_A)$ and $(X_B,\sigma_B)$.
It satisfies $h \circ\sigma_A = \sigma_B\circ h$.
As $h:X_A\longrightarrow X_B$ gives rise to a continuous orbit equivalence
between them in the sense of \cite{MaPacific},
a homomorphism $\varPsi_h: C(X_B,\Z) \longrightarrow C(X_A,\Z)$
of abelian groups is defined by setting 
\begin{equation}\label{eq:Psihf}
\varPsi_h(f)(x) = \sum_{i=0}^{l_1(x)}f(\sigma_B^i(h(x))) -
 \sum_{j=0}^{k_1(x)}f(\sigma_B^j(h(\sigma_A(x))),
\qquad f \in C(X_B,\Z), \, x \in X_A,
\end{equation} 
where $k_1(x),  l_1(x)$ are nonnegative integers
satisfying the equation
\begin{equation}
\sigma_B^{k_1(x)} (h(\sigma_A(x))) 
 = \sigma_B^{l_1(x)}(h(x)) \quad \text{ for } \quad 
x \in X_A.  \label{eq:orbiteq1x} 
\end{equation}
By \cite[Theorem 3.2]{MaMZ2017}, there exists 
 an isomorphism
$\Phi: \OA\longrightarrow \OB$ of $C^*$-algebras such that 
\begin{equation}\label{eq:Psih}
\Phi(\DA) = \DB
\quad
\text{ and }
\quad
 \Phi\circ\rho^{A, \varPsi_h(f)}_t = \rho^{B, f}_t\circ\Phi
\quad
\text{ for all } f \in C(X_B,\Z), \,  t \in \T.
\end{equation}
Now $h:X_A\longrightarrow X_B$ is a topological conjugacy, so that 
one may take the integers such as 
$k_1(x)= 0, l_1(x) = 1$ for all $x \in X_A$. 
Hence we know $\varPsi_h(f) = f\circ h$,
proving the assertion (ii).

(ii) $\Longrightarrow$ (i):
Assume that there exists  an isomorphism
$\Phi: \OA\longrightarrow \OB$ of $C^*$-algebras satisfying 
$\Phi(\DA) = \DB$ and the equalities \eqref{eq:gauge}.
Since 
the isomorphism
$\Phi: \OA\longrightarrow \OB$
satisfies $\Phi(\DA) = \DB$,
the homeomorphism
$h: X_A \longrightarrow X_B$ 
satisfying
$\Phi(a) = a\circ h^{-1}$ under the canonical identification
between $\DA$ and $C(X_A)$
gives rise to a continuous orbit equivalence between 
$(X_A,\sigma_A)$ and $(X_B,\sigma_B)$
(\cite[Proposition 5.3 and Proposition 5.5]{MaPacific}).
Hence as in  \cite[Theorem 3.2]{MaMZ2017},
the homeomorphism 
$h: X_A \longrightarrow X_B$ 
extends to the whole $C^*$-algebra 
$\OA$, so that 
there exists an  isomorphism
$\Phi_1: \OA\longrightarrow \OB$ of $C^*$-algebras such that 
\begin{equation}\label{eq:Phi1}
\Phi_1(\DA) = \DB
\quad
\text{ and }
\quad
 \Phi_1\circ\rho^{A, \varPsi_h(f)}_t = \rho^{B, f}_t\circ\Phi_1
\quad
\text{ for all } f \in C(X_B,\Z), \,  t \in \T,
\end{equation}
and $\Phi_1 (a) = a\circ h^{-1}$ for $a \in \DA$ 
under the canonical identification
between $\DA$ and $C(X_A)$.
The condition  $\Phi_1 (a) = a\circ h^{-1}$ for $a \in \DA$
follows from the construction of $\Phi_1:\OA \longrightarrow \OB$ in \cite{MaMZ2017}.  
Since the original isomorphism
$\Phi: \OA\longrightarrow \OB$ satisfies the condition
$\Phi(\DA) = \DB$ and $\Phi(a) = a\circ h^{-1}, a \in \DA$,
the restriction of the automorphism
$\Phi_1^{-1} \circ \Phi$ on $\DA$ is the identity.
By \cite[Lemma 4.6]{MaJOT2000},
one may find a unitary $U_1 \in \DB$ such that 
$\Phi_1 (S_i) = U_1 \Phi (S_i), i=1,2,\dots,N$,
where $S_i, i=1,2,\dots,N$ are the canonical generating partial 
isometries of $\OA$. 
By \eqref{eq:Phi1}, 
we have
\begin{equation*}
 \Phi_1\circ\rho^{A, \varPsi_h(f)}_t(S_i) = \rho^{B, f}_t\circ\Phi_1(S_i)
 \quad
\text{ for } f \in C(X_B,\Z), \,  t \in \T.
\end{equation*}
Since $\rho^{A, \varPsi_h(f)}_t(S_i) = \exp({2\pi\sqrt{-1}t \varPsi_h(f)}) \cdot S_i$
we have
\begin{equation*}
\Phi_1(\exp({2\pi\sqrt{-1}t \varPsi_h(f)} ))\cdot \Phi_1(S_i) = \rho^{B, f}_t(U_1 \Phi(S_i)).
\end{equation*}
As $\Phi_1(\exp({2\pi\sqrt{-1} t \varPsi_h(f)} )) = \Phi(\exp({2\pi\sqrt{-1}t \varPsi_h(f)} ))$
because $\exp({2\pi\sqrt{-1} t \varPsi_h(f)})\in \DA$,
we have
\begin{equation*}
\Phi(\exp({2\pi\sqrt{-1} t \varPsi_h(f)} ))\cdot U_1 \Phi(S_i) = U_1 \rho^{B, f}_t(\Phi(S_i))
\end{equation*}
and hence
\begin{equation*}
\Phi(\exp({2\pi\sqrt{-1} t \varPsi_h(f)} ))\cdot \Phi(S_i) = \rho^{B, f}_t(\Phi(S_i))
\end{equation*}
so that 
\begin{equation*}
\Phi(\rho^{A,\varPsi_h(f)}_t (S_i)) = \rho^{B, f}_t(\Phi(S_i)).
\end{equation*}
This implies that 
the equality
\begin{equation}\label{eq:PhiPsih}
\Phi\circ\rho^{A, \varPsi_h(f)}_t = \rho^{B, f}_t\circ\Phi \quad\text{ for all } f \in C(X_B,\Z) 
\end{equation}
hold.
By \eqref{eq:gauge} and \eqref{eq:PhiPsih},
we have
\begin{equation}\label{eq:varPsihfh}
\varPsi_h(f) = f\circ h \qquad \text{ for all } f \in C(X_B,\Z).
\end{equation}
In \eqref{eq:gauge}, by taking $f\equiv 1$,
we have 
$\Phi(\DA) = \DB$ and $\Phi\circ\rho^A_t = \rho^B_t\circ\Phi, t \in \T$.
Hence 
$(X_A,\sigma_A)$ and $(X_B,\sigma_B)$
are eventually conjugate via the homeomorphism 
$h:X_A\longrightarrow X_B$. 
Hence there exists 
a nonnegative integer $K$ satisfying \eqref{eq:event}.
The final step to complete the proof of the impication (ii) $\Longrightarrow$ (i)
is to show the following lemma.
\end{proof}  
\bigskip
\begin{lemma}
Suppose that $(X_A,\sigma_A)$ and $(X_B,\sigma_B)$ are eventually conjugate 
such that there exists 
a nonnegative integer $K$ satisfying \eqref{eq:event}. 
If the equality \eqref{eq:varPsihfh} holds,
then $h:X_A\longrightarrow X_B$ gives rise to a topological conjugacy between 
$(X_A,\sigma_A)$ and $(X_B,\sigma_B)$.
\end{lemma}
\begin{proof}
Now the nonnegative integer $K$ satisfies \eqref{eq:event},
so that we have by \eqref{eq:Psihf}
\begin{equation*}
\varPsi_h(f)(x) = \sum_{i=0}^{K+1}f(\sigma_B^i(h(x))) -
 \sum_{j=0}^{K}f(\sigma_B^j(h(\sigma_A(x)))),
\qquad f \in C(X_B,\Z), \, x \in X_A.
\end{equation*} 
If $K=0$, the homeomorphism 
$h:X_A\longrightarrow X_B$ gives rise to a topological conjugacy between 
$(X_A,\sigma_A)$ and $(X_B,\sigma_B)$.
Hence we assume that $K \ge 1$.

By the condition \eqref{eq:varPsihfh} together with
the equality
$\sigma_B^{K+1}(h(x)) = \sigma_B^K(h(\sigma_A(x)))$, 
we see the equality
\begin{equation}\label{eq:fk}
\sum_{i=1}^{K}f(\sigma_B^i(h(x))) =
 \sum_{j=0}^{K-1}f(\sigma_B^j(h(\sigma_A(x)))),
\qquad f \in C(X_B,\Z), \, x \in X_A
\end{equation} 
holds.
For a fixed $x \in X_A$, we put
$y=\sigma_B(h(x)), \, w= h(\sigma_A(x))$
so the we obtain the equalities 
$\sigma_B^K(y) = \sigma_B^K(w)
$
and
\begin{equation}\label{eq:fK2}
\sum_{i=0}^{K-1}f(\sigma_B^i(y)) =
 \sum_{j=0}^{K-1}f(\sigma_B^j(w)),
\qquad f \in C(X_B,\Z).
\end{equation} 
Put $y(j) = \sigma_B^j(y), w(j) = \sigma_B^j(w), j=0,1,\dots,K-1$
and
$$
Y_0 = \{ y(0), y(1), \dots, y(K-1)\}, \qquad
W_0 = \{ w(0), w(1), \dots, w(K-1)\}. 
$$
By \eqref{eq:fK2}, we have
\begin{equation}\label{eq:fk3}
\sum_{y(i)\in Y_0}f(y(i)) =
 \sum_{w(j)\in W_0}f(w(j)),
\qquad f \in C(X_B,\Z).
\end{equation} 
If $Y_0 \cap W_0 =\emptyset$,
one may find $f_0\in C(X_B,\Z)$ such that 
$$
f_0(y(i)) =1,\qquad
f_0(w(i) )=0,\quad \text{ for all } i=0, 1,\dots, K-1,
$$ 
a contradiction to \eqref{eq:fK2} unless $K=0$.
Hence 
$Y_0 \cap W_0 \ne \emptyset$. 
Take $i_0, j_0\in \{0,1,\dots,K-1\}$
such that $y(i_0) = w(j_0)$.
We put 
$$
Y_1 = Y_0 \backslash \{y(i_0)\}, \qquad
W_1 = W_0 \backslash \{w(j_0)\}
$$
so that we have
\begin{equation}\label{eq:fk4}
\sum_{y(i)\in Y_1}f(y(i)) =
 \sum_{w(j)\in W_1}f(w(j)),
\qquad f \in C(X_B,\Z).
\end{equation} 
Inductively we finally know that  
 $Y_0 = W_0$ unless $K=0$.
Hence we may find $p, q \in \{ 0,1,\dots, K-1\}$
such that 
$y=\sigma_B^p(w), w=\sigma_B^q(y)$.
If $q=0$, then we have 
$h(\sigma_A(x)) = \sigma_B(h(x))$.
If $q \ne 0$, 
we have $ y= \sigma_B^{p+q}(y)$
and hence $y$ is periodic.
Therefore we conclude that 
the equality 
$h(\sigma_A(x)) = \sigma_B(h(x))$
holds for $x \in X_A$ such that 
$y=\sigma_B(h(x))$ is not periodic.
A point $x \in X_A$ is said to be eventually periodic
if $\sigma_A^L(x)$ is periodic for some nonnegative integer $L$.
The set of non-eventually periodic points 
is dense in the topological Markov shift for an irreducible non-permutation matrix.
Since a continuous orbit equivalence preserves the set of 
eventually periodic points,
we know that  the equality 
$h(\sigma_A(x)) = \sigma_B(h(x))$
holds for all $x \in X_A$. 
\end{proof}

\begin{remark}\label{re:remark1}
The equality \eqref{eq:gauge} is equivalent to the following equality
 \begin{equation}\label{eq:gaugeg}
\Phi\circ\rho^{A, g}_t = \rho^{B, g\circ h^{-1}}_t\circ\Phi
\qquad \text{ for all } \quad g \in C(X_A,\Z), \,  t \in \T.
\end{equation}
\end{remark}

Let $A, B$ be  irreducible, non-permutation matirces with entries in $\{0,1\}.$
As in \cite{MaPacific}, 
one-sided topological Markov shifts $(X_A, \sigma_A)$ 
and $(X_B, \sigma_B)$ are 
said to be {\it continuously orbit equivalent}\/ 
if there exist nonnegative integer valued continuous functions
$k_1, l_1$ on $X_A$ and $k_2, l_2$ on $X_B$ such that  
\begin{align}
\sigma_B^{k_1(x)} (h(\sigma_A(x))) 
& =  \sigma_B^{l_1(x)}(h(x)), \qquad x \in X_A, \label{eq:coe1}\\ 
\sigma_A^{k_2(y)}(h^{-1}(\sigma_B(y))) 
& = \sigma_A^{l_2(y)}(h^{-1}(y)), \qquad y \in X_B. \label{eq:coe2}
\end{align}
If one may take $k_1 \equiv 0, l_1 \equiv 1, k_2 \equiv 0, l_2 \equiv 1,$ 
then the above equalities \eqref{eq:coe1} and \eqref{eq:coe2} 
reduce to the definition that 
$(X_A, \sigma_A)$ and $(X_B, \sigma_B)$ are topologically conjugate.
If one may take $k_1 \equiv K, l_1 \equiv K+1, k_2 \equiv K, l_2 \equiv K+1$
for some constant nonnegative integer $K$, 
then the above equalities \eqref{eq:coe1} and \eqref{eq:coe2} 
reduce to the definition that $(X_A, \sigma_A)$ and $(X_B, \sigma_B)$ 
are eventually conjugate.
If one may take $l_1 - k_1 = 1 + b_1 -b_1 \circ \sigma_A$
and $ l_2- k_2 = 1 + b_2 - b_2\circ \sigma_B$
for some integer valued continuous functions 
$b_1: X_A\longrightarrow \Z$ and $b_2: X_B\longrightarrow \Z$,
respectively, 
then the above equalities \eqref{eq:coe1} and \eqref{eq:coe2} 
reduce to the definition that $(X_A, \sigma_A)$ and $(X_B, \sigma_B)$ 
are {\it strongly continuous orbit equivalent}\/ (\cite{MaJOT2015}).
The continuous orbit equivalence between  
$(X_A, \sigma_A)$ and $(X_B, \sigma_B)$  
is completely characterized by the condition that there exists an isomorphism
$\Phi:\OA\longrightarrow \OB$ satisfying $\Phi(\DA) = \DB$.
Take a homeomorphism $h:X_A\longrightarrow X_B$
such that $\Phi(a) = a \circ h^{-1}$ for $a \in \DA$.
In particular, we see that $\Phi(g) = g \circ h^{-1}$ for $g \in C(X_A,\Z).$
We finally summarize characterization of these subequivalence relations of continuous orbit equivalence 
in one-sided topological Markov shifts in the following way.
\begin{corollary}[{Theorem \ref{thm:main} and \cite[Corollary 3.5]{MaMZ2017}, see also  \cite[Theorem 1.5]{MaPAMS2017}, \cite[Theorem 6.7]{MaJOT2015}}]
Let $\Phi: \OA\longrightarrow \OB$ be an isomorphism of $C^*$-algebras satisfying $\Phi(\DA) = \DB$.
Let $h:X_A\longrightarrow X_B$ be the homeomorphism satisfying $\Phi(a) = a\circ h^{-1}$ for $a \in \DA$. 
\begin{enumerate}
\renewcommand{\theenumi}{\roman{enumi}}
\renewcommand{\labelenumi}{\textup{(\theenumi)}}
\item 
The homeomorphism $h:X_A\longrightarrow X_B$ gives rise to a topological conjugacy between  
$(X_A,\sigma_A)$ and $(X_B,\sigma_B)$ 
 if and only if 
\begin{equation}\label{eq:Phitopcong}
\Phi\circ\rho^{A, g}_t = \rho^{B, \Phi(g)}_t\circ\Phi
\quad\text{ for all }\quad
g \in C(X_A,\Z), \,  t \in \T.
\end{equation}
\item 
The homeomorphism $h:X_A\longrightarrow X_B$ gives rise to an eventual conjugacy between  
$(X_A,\sigma_A)$ and $(X_B,\sigma_B)$ 
 if and only if 
\begin{equation}\label{eq:Phieventcong}
\Phi\circ\rho^{A}_t = \rho^{B}_t\circ\Phi,
\qquad t \in \T.
\end{equation}
\item 
The homeomorphism $h:X_A\longrightarrow X_B$ gives rise to a 
strongly continuous orbit equivalence between  
$(X_A,\sigma_A)$ and $(X_B,\sigma_B)$ 
  if and only if 
there exists a unitary one-cocycle $v_t \in \DB$ for the gauge action 
$\rho^B$ such that 
\begin{equation}\label{eq:Phiscoe}
\Phi\circ\rho^{A}_t = \Ad(v_t) \circ \rho^{B}_t\circ\Phi,
\qquad  t \in \T.
\end{equation}
\end{enumerate}
\end{corollary}
\begin{proof}
(i)
The if part follows from Theorem \ref{thm:main} (ii) $\Longrightarrow$ (i) and its proof
 by noticing Remark \ref{re:remark1}.
We will show the only if part.
Suppose that $h:X_A\longrightarrow X_B$ is a topological conjugacy between  
$(X_A,\sigma_A)$ and $(X_B,\sigma_B)$.
By Theorem \ref{thm:main} (i) $\Longrightarrow$ (ii) and its proof,
one may find an isomorphism $\Phi_1: \OA\longrightarrow \OB$
of $C^*$-algebras such that 
$\Phi_1(\DA) = \DB$, $\Phi_1(a) = a\circ h^{-1}$ for $a \in \DA$ and
\begin{equation}\label{eq:corPhi1}
\Phi_1\circ\rho^{A, g}_t = \rho^{B, \Phi_1(g)}_t\circ\Phi_1
\quad\text{ for all }\quad
g \in C(X_A,\Z), \,  t \in \T.
\end{equation}
Hence $\Phi_1$ coincides with $\Phi$ on the subalgebra $\DA$.
By using a similar argument to the proof of Theorem \ref{thm:main} (ii) $\Longrightarrow$ (i),
one may find a unitary $U_1$ in $\DB$ such that    
$\Phi_1 (S_i) = U_1 \Phi (S_i), i=1,2,\dots,N$,
where $S_i, i=1,2,\dots,N$ are the canonical generating partial 
isometries of $\OA$, so that 
we have
$\Phi\circ\rho^{A, g}_t = \rho^{B, \Phi_1(g)}_t\circ \Phi
$ 
by the same argument as the one obtained from \eqref{eq:Phi1}
to \eqref{eq:PhiPsih}.
As  
$\Phi_1(g) = \Phi(g)$, 
we conclude the equality
\eqref{eq:Phitopcong}.

(ii), (iii)
The  if parts  of (ii) and (iii) 
follow from 
\cite[Corollary 3.5 (i)]{MaMZ2017} and 
\cite[Corollary 3.5 (ii)]{MaMZ2017} 
(see also \cite[Theorem 3.3 (i)]{MaMZ2017} and 
          \cite[Theorem 3.3 (ii)]{MaMZ2017})
and their proofs, respectively.
The only if parts follow from 
\cite[Theorem 3.3 (i)]{MaMZ2017} and \cite[Theorem 3.3 (ii)]{MaMZ2017}
and their proofs, respectively,
by using a similar argument to the only if part of the above proof (i).

%
 \end{proof}

A generalization of Theorem \ref{thm:main} to more general subshifts treated in the paper
\cite{MaPre2020} will be studied in a forthcoming paper \cite{MaPre2020c}.
\bigskip

{\it Acknowledgment:}
The author would like to deeply thank the referee for his/her careful reading the first version of the paper
and various helpful comments in the presentation of the paper.
Especially the referee kindly pointed out an error of the formulation of Corollary 4 in the first version,
so that the statement of Corollary 4 is corrected.
This work was supported by JSPS KAKENHI 
Grant Number 19K03537.


\begin{thebibliography}{99}










\bibitem{BC}
{\sc K. A. Brix and T. M. Carlsen}
{\it Cuntz-Krieger algebras and one-sided conjugacy of shifts of finite type 
and their groupoids}, 
J. Aust. Math. Soc. (2019), 1-10,
doi:10:1017/S1446788719000168, arXiv:1712.00179 [mathOA].

\bibitem{BCW}
{\sc N. Brownlowe, T. M. Carlsen and M. F. Whittaker}, 
{\it Graph algebras and orbit equivalence}, 
Ergodic Theory  Dynam. Systems {\bf 37}(2017), 389--417.


\bibitem{CEOR}
{\sc T. M. Carlsen,  S. Eilers, E. Ortega and G. Restorff},
{\it Flow equivalence and orbit equivalence for shifts of finite type 
and isomorphism of their groupoids}, 
J. Math. Anal. Appl. {\bf 469}(2019), pp. \ 1088--1110. 



\bibitem{CR}
{\sc T. M. Carlsen and J. Rout},
{\it Diagonal-preserving gauge invariant isomorphisms of graph $C^*$-algebras},
J. Funct. Anal. {\bf 273}(2017). pp. \ 2981--2993.


\bibitem{CRS}
{\sc T. M. Carlsen, E. Ruiz and A. Sims},
{\it Equivalence and stable isomorphism of groupoids,
and diagonal-preserving stable isomorphisms of graph $C^*$-algebras 
and Leavitt path algebras},
Proc. Amer. Math. Soc. {\bf 145}(2017), pp. \ 1581--1592.
 
















\bibitem{CK}{\sc J. ~Cuntz and W. ~Krieger},
{\it A class of $C^*$-algebras and topological Markov chains},
 Invent.\ Math.\
 {\bf 56}(1980), pp.\ 251--268.












\bibitem{Kitchens}{\sc B.~P. ~Kitchens},
{\it Symbolic dynamics}, 
Springer-Verlag, Berlin, Heidelberg and New York
(1998).






\bibitem{LM}{\sc D. ~Lind and B. ~Marcus},
{\it An introduction to symbolic dynamics and coding},
 Cambridge University Press, Cambridge
(1995).




\bibitem{MaJOT2000}{\sc K. Matsumoto},
{\it On automorphisms of $C^*$-algebras associated with subshifts},
J. Operator Theory {\bf 44}(2000), pp.\  91--112.




\bibitem{MaPacific}
{\sc K. Matsumoto},
{\it Orbit equivalence of topological Markov shifts and Cuntz-Krieger algebras},
Pacific J.\ Math.\ 
{\bf 246}(2010),
pp.\  199--225.







\bibitem{MaJOT2015}
{\sc K. Matsumoto},
{\it Strongly continuous orbit equivalence of 
one-sided topological Markov shifts},
J. Operator Theory {\bf 74}(2015), pp. 101--127.



\bibitem{MaPAMS2017}{\sc K. Matsumoto},
{\it Uniformly continuous orbit equivalence of Markov shifts 
and gauge actions on Cuntz--Krieger algebras},
Proc. Amer. Math. Soc. {\bf 145}(2017), pp.\ 1131--1140. 




\bibitem{MaMZ2017}
{\sc K. Matsumoto},
{\it Continuous orbit equivalence, 
flow equivalence of Markov shifts and circle actions on Cuntz--Krieger algebras},
 Math. Z. {\bf 285}(2017), \ 121--141.  

\bibitem{MaTAMS2018}{\sc K. Matsumoto},
{\it Relative Morita equivalence of Cuntz--Krieger algebras and flow equivalence of 
topological Markov shifts},
Trans. Amer. Math. Soc. {\bf 370}(2018), pp.\ 7011--7050. 


\bibitem{MaDyn2019}{\sc K. Matsumoto},
{\it State splitting, strong shift equivalence 
and stable isomorphism of Cuntz-Krieger algebras},
 Dyn. Syst. {\bf 34}(2019), pp. \  93--112.

\bibitem{MaPre2020}
{\sc K. Matsumoto},
{\it Simple purely infinite  $C^*$-algebras associated with normal subshifts},
preprint,  arXiv:2003.11711v2 [mathOA].

\bibitem{MaPre2020c}
{\sc K. Matsumoto},
{\it One-sided topological conjugacy of normal subshifts and 
gauge actions on the associated $C^*$-algebras} (tentaive title), 
in preparation.




\bibitem{MMKyoto}
{\sc K. Matsumoto and H. Matui},
{\it Continuous orbit equivalence of topological Markov shifts 
and Cuntz--Krieger algebras},
Kyoto J. Math. {\bf 54}(2014), pp.\ 863--878.


\bibitem{MMETDS}
{\sc K. Matsumoto and H. Matui},
{\it Continuous orbit equivalence of topological Markov shifts 
and dynamical zeta functions}, 
Ergodic Theory Dynam. Systems
{\bf 36}(2016), pp. \ 1557--1581.









































\end{thebibliography}
\end{document}